\documentclass[reqno]{amsart}
\usepackage{amsmath,amsbsy,amsfonts}
\usepackage{color}
\usepackage[notref,notcite]{showkeys}
\usepackage{hyperref}
\date{}
\def\nd{\noindent}
\def\thend{\rule{3mm}{3mm}}

\newtheorem{theorem}{Theorem}[section]

\newtheorem{prop}{Proposition}[section]
\newtheorem{lem}{Lemma}[section]
\newtheorem{rmk}{Remark}[section]

\newcommand{\w}{W_0^{1,\Phi}(\Omega)}

\newcommand{\R}{\mathbb{R}}

\begin{document}
\title{Multiplicity of solutions for quasilinear elliptic problems}
\vspace{1cm}

\author{M. L. M. Carvalho}
\address{M. L. M. Carvalho \newline Universidade Federal de Goi´as, IME, Goi\^ania-GO, Brazil }
\email{\tt marcos$\_$leandro$\_$carvalho@ufg.br}

\author{J. V. Gon\c calves}
\address{J. V. Gon\c calves \newline Universidade Federal de Goi´as, IME, Goi\^ania-GO, Brazil }
\email{\tt jv@ufg.br}

\author{Edcarlos D. da Silva}
\address{Edcarlos D da Silva \newline  Universidade Federal de Goi´as, IME, Goi\^ania-GO, Brazil}
\email{\tt edcarlos@ufg.br}

\author{K. O. Silva}
\address{K. O. Silva \newline Universidade Federal de Goiás, IME, Goi\^ania-GO, Brazil}
\email{\tt kaye$_0$@ufg.com}

\thanks{The third author was partially supported by Cnpq/Brazil. }

\begin{abstract}
It is established existence, uniqueness and multiplicity of solutions for a quasilinear elliptic problem problems driven by $\Phi$-Laplacian operator. Here we consider the reflexive and nonreflexive cases using an auxiliary problem. In order to prove our main results we employ variational methods, regularity results and truncation techniques.
\end{abstract}

\maketitle

\section{Introduction}
The present work concerns  existence, uniqueness and multiplicity of solutions for the elliptic problems
\begin{equation}\label{p1}
\left\{\
\begin{array}{l}
\displaystyle-\Delta_\Phi u=f(x),~\mbox{in}~\Omega,\\ \\
u=0~\mbox{on}~\partial \Omega
\end{array}
\right.
\end{equation}
and
\begin{equation}\label{p2}
\left\{\
\begin{array}{l}
\displaystyle-\Delta_\Phi u=g(x,u),~\mbox{in}~\Omega,\\ \\
u=0~\mbox{on}~\partial \Omega
\end{array}
\right.
\end{equation}
where $\Omega\subset\mathbb{R}^N,~N\geq 2$, is a bounded domain with smooth boundary, $\Phi$ is the even function defined by
$$\Phi(t)=\int_0^ts\phi(s)ds,~t\in\mathbb{R}.$$

Quasilinear elliptic problems driven by the $\Phi$-Laplacian operator have been widely considered in the last years. Here we infer the reader to \cite{chung,hui-2,Fuchs1,Fuchs2,Fuk_1,Fuk_2,hui-3,gossez-Czech,Gz1,Le2000,MI,MugnaiPapageorgiou,fang}. Most of them considered the Orlicz and Orlicz-Sobolev framework taking into account that $\Phi$ and $\widetilde{\Phi}$ verify the so called $\Delta_{2}$-condition. Under this condition the Orlicz and Orlicz-Sobolev space are separable and reflexive Banach spaces. The main novel in this work is to consider quasilinear elliptic problem such as \eqref{p1} and \eqref{p2} where the condition $\Delta_{2}$ condition is not available anymore. The main difficulty here is to consider the weak star convergence instead of the weak converge for the Orlicz-Sobolev space. The aproach here is variational using an energy functional associated to the elliptic problems \eqref{p1} or \eqref{p2}. At the same time,  we consider a regularity result finding existence and multiplicity of solutions for quasilinear elliptic equations without the $\Delta_{2}$ condition. In our setting we consider an auxiliary problem in order to recover some compactness for our energy functional which is crucial in variational methods.

Now we shall give the hypotheses for the functions $\phi, f$ and $g$. For the function $\phi:\mathbb{R}\rightarrow\mathbb{R}$  we assume that $\phi$ is in $C^1$ and it satisfies
\begin{itemize}
	\item[($\phi_1$)]  \     \ $\mbox{(i)} \ \ t\phi(t)\to 0 \ \mbox{as} \  t\to 0, \qquad\qquad
	~~ \mbox{(ii)} \  t\phi(t)\to\infty \ \mbox{as} \  t\to\infty$;
	\item[($\phi_2$)]  \ $t\phi(t) \ \mbox{\it is strictly increasing in}~ (0, \infty)$;
	\item[($\phi_3$)] there exist $\ell,m\in [1,N)$ such that
	$$\ell -1 =\inf_{t>0}\frac {(t\phi(t))^\prime}{\phi(t)}\leq \frac {(t\phi(t))^\prime}{\phi(t)}\leq m-1,~t>0;$$
	\item[($\phi_4$)] $a:=\displaystyle\inf_{t>0}\frac{t^m}{\Phi(t)}>0$.
\end{itemize}
Moreover we assume that
\begin{equation}\label{cond-f}
	f\in L^N(\Omega).
\end{equation}
For the function $g:\Omega\times \mathbb{R}\rightarrow\mathbb{R}$ we suppose that $g$ is in $C^{0}$ class and $g(x,0)=0$ for any $x\in\Omega$. Furthermore, we assume also the following assumptions:
\begin{description}
	\item[$(g_1)$] there are a  constant $C>0$ and  an  N-function
\end{description}
$$
\Psi(t)=\int_0^t\psi(s)ds
$$
\nd with $\psi:[0, \infty)\to\mathbb{R}$ continuous and  satisfying
\begin{description}
	\item[$(\psi_1)$]~~~~~$\displaystyle 1< m<\ell_\Psi:=\inf_{t>0}\frac{t\psi(t)}{\Psi(t)}\leq
	\sup_{t>0}\frac{t\psi(t)}{\Psi(t)}=:m_\Psi<1^*:=\frac{ N}{N-1}$;
\end{description}
\nd  such that
\begin{equation*}
|g(x,t)|\leq C(1+\psi(t)),
\end{equation*}
\begin{description}
	\item[$(g_2)$]  there is an  an N-function
\end{description}
$$
\Gamma(t)=\int_0^t\gamma(s)ds
$$
\nd with $\gamma:[0, \infty) \to \R$ continuous and satisfying
\begin{description}
	\item[$(\gamma_1)$]~~~~~   $\displaystyle  N<\ell_\Gamma:=\inf_{t>0}\frac{t\gamma(t)}{\Gamma(t)}
	\leq \sup_{t>0}\frac{t\gamma(t)}{\Gamma(t)}=:m_\Gamma<\infty$,
\end{description}
\nd  such that
\begin{equation*}
\Gamma\left(\frac{G(x,t)}{|t|^\ell}\right)\leq C\overline{G}(x,t), x \in \Omega,~~|t|\geq R,
\end{equation*}
\nd  where $C, R$ are positive constants,
$$G(x,t):=\int_0^tg(x,s)ds$$
and
$$
\overline{G}(x,t):= tg(x,t)-mG(x,t), x \in \Omega, t \in \mathbb{R}.
$$
\nd Here we denote $\lambda_{1} > 0$ the first eigenvalue for the operator $\Delta_{\Phi}$.  Recall that, using hypothesis $(\phi_3)$, it folows from the Poincar\'e inequality, (see e.g.  \cite{hui-2}, \cite{gossez-Czech}), that
	\begin{equation*}\label{lambda1}
	\lambda_{1} \int_\Omega \Phi(u) dx \leq \int_\Omega \Phi(|\nabla u|)dx,~u \in  \w.
	\end{equation*}
Now we shall consider some additional hypotheses:
\begin{description}
	\item[$(g_3)$]~~~~~~~~~~~~~~~~~~~ $\displaystyle \lim_{t\rightarrow\infty }\frac{g(x,t)}{|t|^{m-1}}=\infty$,
\end{description}
\begin{description}
	\item[$(g_4)$]~~~~~~~~~~~~~~~~~~~~ $\displaystyle \limsup_{t\rightarrow 0 }\frac{g(x,t)}{|t|\phi(t)}=\lambda< \lambda_1$,
\end{description}
Due to the nature of the operator
$$\Delta_\Phi u=\mbox{div}(\phi(|\nabla u|)\nabla u)$$
we need to consider the framework of Orlicz-Sobolev spaces $\w$. It is important to emphasize that $\Phi$-Laplacian operator is not homogenous. This is a serious difficulty in order to use variational methods. In order to overcome this difficulty we shall consider some specific estimates in Orlicz and Orlicz-Sobolev spaces.
\begin{rmk}
Taking into account hypothesis $(\phi_3)$ we have that $t\mapsto \frac{t^m}{\Phi(t)}$ is a strictly increasing function. As a consequence we mention that
$$\inf_{t>0}\frac{t^m}{\Phi(t)}=\lim_{t\rightarrow 0_+}\frac{t^m}{\Phi(t)}.$$
\end{rmk}
Our principal result can be stated in the following form
\begin{theorem}\label{th1}
	Assume $(\phi_1)-(\phi_4)$ and \eqref{cond-f}. Then there exists an unique solution for the elliptic problem \eqref{p1}, that is, there exists an unique function $u\in\w$ in such way that
	\begin{equation}\label{sol-fraca2}
		\int_{\Omega}\phi(|\nabla u|)\nabla u\nabla vdx=\int_\Omega fvdx,~v\in\w.
	\end{equation}
Moreover, asssuming that $\ell > 1$ the solution given just above is in also in $L^{\infty}(\Omega)$ whenever the function $\Phi$ is equivalent to the function $t \rightarrow |t|^{r}$ for some $r > 1$, i.e, there exist $c_{1}, c_{2} > 0$ in such way that $c_{1}|t|^{r} \leq \Phi(t) \leq c_{2}|t|^{r}$ for any $ t \in \mathbb{R}$.
\end{theorem}

We point out that the function
$$\phi(t)=\frac{\log(1+|t|)}{|t|},~t \in \mathbb{R}\backslash\{0\}$$
satisfies the hypotheses $(\phi_1)-(\phi_4)$. In this case the operator in problem \eqref{p1} has logarithmic growth with respect to the gradient which can be written in the following form
\begin{eqnarray}\label{bb}
\left\{\
\begin{array}{l}
\displaystyle-\mbox{div}\left(\frac{\log(1+|\nabla u|)}{|\nabla u|}\nabla u\right)=f~\mbox{in}~\Omega,\\ \\
u=0~\mbox{on}~\partial \Omega.
\end{array}
\right.
\end{eqnarray}
Here we stress out that $\ell = 1$ and $a=m=2$ which give us an concrete example where the N-Function $\Phi$ is in such way that $\tilde{\Phi}$ does not verity the well known $\Delta_{2}$ condition at infinity, see \cite{Rao1}. As a consequence the N-function $\Phi : \mathbb{R} \rightarrow \mathbb{R}$ given by
\begin{equation*}
\Phi(t) = \int_{0}^{t} \log(1+ s) ds, s > 0
\end{equation*}
is in such way that $W^{1,\tilde{\Phi}}_{0}(\Omega)$ is not reflexive. The problem \eqref{bb} have been studied by many authors during the last years, see Boccardo et al \cite{Boccardo1,Boccardo2}, Esposito et al \cite{esposito1}, Passarelli \cite{Passarelli}, Fuchs \cite{Fuchs1,Fuchs2}, Zhang et al \cite{Zhang} and references therein. For further results on Orlicz and Orlicz-Sobolev framework in refer the reader to \cite{A,Fuk_1,Fuk_2,gossez-Czech,Gz1,MI}. The main feature in this work is to find a weak solution for the problem \eqref{p1} in the nonreflexive case using a sequence of approximate problems where in each term for this sequence the associated Orlicz-Sobolev space is reflexive. So that taking the limit in this sequence the solution for the nonreflexive case is obtained by a careful analysis on continuous and compact embedding involved in Orlicz-Sobolev spaces.

For the next result we shall consider the nonlinear elliptic problem \eqref{p2} under some superlinear conditions at infinity. The main feature here is to consider nonreflexive problems without the well known Ambrosetti-Rabinowitz condition at infinity. Namely, the Ambrosetti-Rabinowitz condition, for the function $g$, in short $(AR)$ condition, says that
\begin{equation*}
0 < \theta G(x,t) \leq t g(x, t) , x \in \Omega, |t| \geq R
\end{equation*}
holds true for some $\theta > m$ and $R > 0$. As a product the $(AR)$ condition implies that
\begin{equation}\label{cc}
G(x,t) \geq c_{1} |t|^{m} - c_{2}, x \in \Omega, t \in \mathbb{R}
\end{equation}
holds for some $c_{1}, c_{2} > 0.$
Nevertheless, there are superlinear functions in such way that \eqref{cc} in not satisfied. For example, we mention that $g(x,t) = |t|^{m -2} t ln (1 + |t|)$ does not verity the superlinear condition given in \eqref{cc} for each $m \in (1, N)$. As a consequence the function $g$ just above does not verify the $(AR)$ condition.  We point out that the $(AR)$ condition implies some compactness properties such as the Palais-Smale condition at infinity which is crucial in variational methods. As the $(AR)$ condition is not available in our setting we need to consider some compactness condition such as the Cerami condition. Latter on, we shall give a precise definition for the Palais-Smale condition and Cerami condition. For the next result we shall consider hypotheses $(g_{1})-(g_{4})$ proving that the associated functional for the problem \eqref{p2} satisfies the well known Cerami condition which is sufficient in variational procedures. Our second result can be read in the following form

\begin{theorem}\label{th2}
	Assume $(\phi_1),(\phi_2),(\phi_4),$
		\begin{description}
			\item[$(\phi_3)^\prime$]$\displaystyle 1 \leq \ell \leq \frac{\phi(t)t^2}{\Phi(t)} \leq m, \,\,t>0,$
		\end{description}
	and $(g_{1})- (g_{4})$. Then problem \eqref{p2} at least one solution $u \in \w $, that is, there exists a function $u\in\w$  in such way that
	\begin{equation}\label{sol-fraca}
		\int_{\Omega}\phi(|\nabla u|)\nabla u\nabla vdx=\int_\Omega g(x, u) vdx,~v\in\w.
	\end{equation}
Assuming $(\phi_{3})$ instead of $(\phi_{3})^{\prime}$, problem \eqref{p2} admits at least two weak solutions $u_{1}, u_{2}  \in \w /\{0\}$ satisfying $u_{1} \geq 0 $ and $u_{2} \leq 0$ in $\Omega$. Furthermore, assuming also that $\ell > 1$ and the function $\Phi$ is equivalent to $t \rightarrow |t|^{r}, t \in \mathbb{R}$ for some $r > 1$, the solutions $u_{1}, u_{2}$ belong to $C^{1,\alpha}$ for some $\alpha \in (0,1)$.
\end{theorem}

The paper is organized as follows: Section 2 is devoted to an overview on Orlicz and Orlicz-Sobolev framework. In Section 3 we consider the elliptic problem \eqref{p1} in the reflexive case. In Section 4 we give some existence results for the problem \eqref{p1} in the nonreflexive case. Section 5 is devoted to regularity results to the elliptic problem \eqref{p1} and \eqref{p2}. In Section 5 we give the proof of our main results.

\section{Orlicz and Orlicz-Sobolev spaces}
\nd The reader is  referred to  \cite{A,DT,Rao1} regarding Orlicz-Sobolev spaces.  The usual norm on $L_{\Phi}(\Omega)$ is (Luxemburg norm),
\[
\|u\|_\Phi=\inf\left\{\lambda>0~\big|~\int_\Omega \Phi\left(\frac{u(x)}{\lambda}\right) dx \leq 1\right\},
\]
\nd the  Orlicz-Sobolev norm of $ W^{1, \Phi}(\Omega)$ is
\[
\displaystyle \|u\|_{1,\Phi}=\|u\|_\Phi+\sum_{i=1}^N\left\|\frac{\partial u}{\partial x_i}\right\|_\Phi,
\]
\nd while $\w$ denotes the closure of $C_0^{\infty}(\Omega)$ with respect to the usual norm of $W^{1,\Phi}(\Omega)$. Recall that
$$
\widetilde{\Phi}(t) = \displaystyle \max_{s \geq 0} \{ts - \Phi(s) \},~ t \geq 0.
$$
\nd It turns out that when $\Phi$ and $\widetilde{\Phi}$  are  $N$-functions  satisfying the $\Delta_2$-condition we mention that $L_{\Phi}(\Omega)$  and $W^{1,\Phi}(\Omega)$  are separable, reflexive,  Banach spaces, see \cite[p 22]{Rao1}. However, as was quoted in the introduction we shall consider the case when the function $\Phi$ does not verify the $\Delta_{2}$-condition. Anyway, we shall consider some important properties for Orlicz and Orlicz-Sobolev spaces.

\begin{rmk}\label{reflex}
	It is well known that $(\phi_3)$ implies that $(\phi_3)^\prime$ is verified. Furthermore, assuming $1 < \ell \leq m < N$, we obtain $\Phi,\widetilde\Phi\in\Delta_2$. Reciprocally, assuming that
$\Phi,\widetilde\Phi\in\Delta_2$ then $1 < \ell \leq m < N$.
\end{rmk}

\nd By the Poincar\'e Inequality, (see e.g.  \cite{gossez-Czech}),
\[
\int_\Omega\Phi(u)dx\leq \int_\Omega\Phi(2d|\nabla u|)dx
\]
\nd where $d=\mbox{diam}(\Omega)$, and it follows that
\[
\|u\|_\Phi\leq 2d\|\nabla u\|_\Phi~\mbox{for}~ \w.
\]
\nd As a consequence,  $\|u\| :=\|\nabla u\|_\Phi$ defines a norm in $\w$, equivalent to $\|.\|_{1,\Phi}$. Let $\Phi_*$ be the inverse of the function
$$
t\in(0,\infty)\mapsto\int_0^t\frac{\Phi^{-1}(s)}{s^{\frac{N+1}{N}}}ds
$$
\nd which extends to ${\R}$ by  $\Phi_*(t)=\Phi_*(-t)$ for  $t\leq 0.$
We say that a N-function $\Psi$ grow essentially more slowly than $\Phi_*$, we write $\Psi<<\Phi_*$, if
$$
\lim_{t\rightarrow \infty}\frac{\Psi(\lambda t)}{\Phi_*(t)}=0,~~\mbox{for all}~~\lambda >0.
$$

The imbedding below (cf. \cite{A}) will be  used in this paper:
$$
\displaystyle \w \stackrel{\tiny cpt}\hookrightarrow L_\Psi(\Omega),~~\mbox{if}~~\Psi<<\Phi_*,
$$
in particular, as $\Phi<<\Phi_*$ (cf. \cite[Lemma 4.14]{Gz1}),
$$
\w \stackrel{\tiny{cpt}} \hookrightarrow L_\Phi(\Omega).
$$
Furthermore,
$$
W_0^{1,\Phi}(\Omega) \stackrel{\mbox{\tiny cont}}{\hookrightarrow} L_{\Phi_*}(\Omega).
$$
It is worthwhile to mention that under hypotheses $(\phi_1)-(\phi_2)$ and $(\phi_3)$  (cf. \cite[Lem. D.2]{clement}) the following continuous
embedding holds
$$L^{m}(\Omega) \stackrel{\mbox{\tiny cont}}\hookrightarrow L_\Phi(\Omega)\stackrel{\mbox{\tiny cont}}\hookrightarrow L^\ell(\Omega).$$

Now we refer the reader to \cite{Fuk_1, Fuk_2} for some elementary results on Orlicz and Orlicz-Sobolev spaces.
\vskip.2cm
\begin{prop}\label{lema_naru}
       Assume that  $\phi$ satisfies  $(\phi_1)-(\phi_3)$.
        Set
         $$
         \zeta_0(t)=\min\{t^\ell,t^m\},~~~ \zeta_1(t)=\max\{t^\ell,t^m\},~~ t\geq 0.
        $$
  \nd Then  $\Phi$ satisfies
       $$
            \zeta_0(t)\Phi(\rho)\leq\Phi(\rho t)\leq \zeta_1(t)\Phi(\rho),~~ \rho, t> 0,
        $$
$$
\zeta_0(\|u\|_{\Phi})\leq\int_\Omega\Phi(u)dx\leq \zeta_1(\|u\|_{\Phi}),~ u\in L_{\Phi}(\Omega).
 $$
    \end{prop}
\begin{prop}\label{lema_naru_*}
    Assume that  $\phi$ satisfies $(\phi_1)-(\phi_3)$.  Set
    $$
    \zeta_2(t)=\min\{t^{\ell^*},t^{m^*}\},~~ \zeta_3(t)=\max\{t^{\ell^*},t^{m^*}\},~~  t\geq 0
    $$
\nd where $1<\ell,m<N$ and $m^* = \frac{mN}{N-m}$, $\ell^* = \frac{\ell N}{N-\ell}$.  Then
        $$
            \ell^*\leq\frac{t^2\Phi'_*(t)}{\Phi_*(t)}\leq m^*,~t>0,
       $$
        $$
            \zeta_2(t)\Phi_*(\rho)\leq\Phi_*(\rho t)\leq \zeta_3(t)\Phi_*(\rho),~~ \rho, t> 0,
       $$
       $$
            \zeta_2(\|u\|_{\Phi_{*}})\leq\int_\Omega\Phi_{*}(u)dx\leq \zeta_3(\|u\|_{\Phi_*}),~ u\in L_{\Phi_*}(\Omega).
        $$
    \end{prop}

\section{The problem \eqref{p1} for the reflexive case}

In this section we shall prove some an existence and uniquiness results for problem \eqref{p1} in the reflexive case using hypotheses $(\phi_1)-(\phi_2)$ and $(\phi_3)^\prime$. In other words, we shall consider hypotheses $(\phi_1)-(\phi_2)$ and $(\phi_3)^\prime$ where $\ell>1$.
Under these conditions it is well known that Orlicz and Orlicz-Sobolev spaces are Banach reflexive spaces. In this way, we shall develop some minimization arguments on Orlicz-Sobolev spaces for a specific functional. It is worthwhile to mention that \eqref{p1} have been considered by Gossez \cite{gossez-Czech}. Here we also refer the reader to Fukagai et al. \cite{Fuk_1,Fuk_2}. For the reader convenience we shall give here an alternative proof.

Consider the energy functional $I:\w\rightarrow\mathbf{R}$ given by
\begin{equation*}\label{energia}
I(u)=\displaystyle\int_\Omega\Phi(|\nabla u|)-fudx,~u\in\w.
\end{equation*}
\begin{prop}\label{prop_ltda_inf}
Suppose $(\phi_1)-(\phi_2),~(\phi_3)^\prime$ and $\ell>1$. Then the functional  $I$ is bounded from below.
\end{prop}
\begin{proof}
Initially, we mention that H\"older's inequality and the embedding
$$\w\hookrightarrow W^{1,1}_0(\Omega)\hookrightarrow L^{\frac{N}{N-1}}(\Omega)$$
imply that
$$
\begin{array}{lll}
I(u) & \geq & \displaystyle\int_\Omega\Phi(|\nabla u|)dx-\|f\|_N\|u\|_{\frac{N}{N-1}} \\
& \geq & \displaystyle\int_\Omega\Phi(|\nabla u|)dx-C\|f\|_N\|u\|\\
& \geq & \|u\|\left(\displaystyle \frac{\displaystyle\int_\Omega\Phi(|\nabla u|)dx}{\| u\|}-C\|f\|_N\right),~u\in\w.
\end{array}
$$
According to Gossez \cite{Gz1}) there exists $M>0$ in such way that $\| u\|>M$ implies
$$\displaystyle \frac{\displaystyle\int_\Omega\Phi(|\nabla u|)dx}{\|\nabla u\|_\Phi}>(C+1)\|f\|_N.$$
As a consequence, we deduce that
$$I(u)\geq\|u\|\|f\|_N\geq M\|f\|_N.$$

On the other hand, using the fact that $I$ continuous and convex, we know that $I$ is weakly lower semicontinous, in short, we say that $I$ is w.l.s.c. Furthermore, the unit ball centered at the origin with radius $M > 0$ given by $B_M(0)=\{u\in\w:~\|u\|_\Phi\leq M\}$ is compact in the weak topology proving that $I$ admits a maximum point in $B_M(0)$. Let $u_0 \in \w$ be fixed such that $I(u_{0}) = \inf \{I(v), v \in B_{M}(0)\}$. Using the estimates discussed just above it follows that
$$I(u)\geq \min\{I(u_0), M\|f\|_N\}.$$
In particular, the functional $I$ is bounded from below. This ends the proof.
\end{proof}

\begin{prop}\label{prop_min}
	Suppose $(\phi_1)-(\phi_2),~(\phi_3)^\prime$ where $\ell>1$. Then there exists $u\in\w$ in such way that
	$$\displaystyle I(u)= \inf_{v\in\w}I(v).$$
\end{prop}
\begin{proof}
Firstly, we observe that Proposition \ref{prop_ltda_inf} shows that $I$ is bounded from below. Let $(u_n)\subseteq\w$ be a minimizer sequence for $I$. Now we claim that $(u_n)$ is bounded in $\w$. In fact, using Proposition \ref{lema_naru}, there exists $n_0 \in \mathbb{N}$ in such way that for any $n>n_0$ implies that
$$
\begin{array}{lllll}
I_\infty+1 & \geq & I(u_n) & \geq & \min\left\{\|u_n\|^\ell,\|u_n\|^m\right\}-\|f\|_{W^{-1,\widetilde{\Phi}}}\|u_n\|,
\end{array}
$$
where we define $I_\infty := \inf_{v\in\w}I(v)$. Hence the sequence $(u_n)$ is now bounded. Using the fact that Orlicz-Sobolev space is a reflexive Banach space we deduce that
$$u_n\rightharpoonup u~~\mbox{em}~~\w.$$
Using the fact that $E: \w \rightarrow \mathbb{R}$ given by
$$E(z)=\int_\Omega\Phi(|\nabla z|)dx,~~z\in\w$$
is w.l.s.c we obtain the following estimates
$$
\begin{array}{lll}
I(u) & \leq & \displaystyle\liminf \int_\Omega\Phi(|\nabla u_n|)dx-\lim_{n\rightarrow\infty} \int_\Omega fu_ndx \leq  \displaystyle\liminf I(u_n) =  I_\infty,
\end{array}
$$
Here we have used the fact that $u_{n} \rightarrow u$ in $L^{N/(N -1)}(\Omega)$. This can be proven using the compact embedding $W^{1,\Phi}_{0}(\Omega) \subset L^{N/(N -1)}(\Omega)$. As a consequence we infer that
$$I(u)=\min_{v\in\w}I(v).$$
The proof for this proposition is now complete.
\end{proof}

\begin{prop}\label{teo_prin}
Suppose $(\phi_1)-(\phi_2),~(\phi_3)^\prime$ where $\ell>1$. Then the problem \eqref{p1} admits at least one weak solution $u\in\w$ which is given by a minimization for $I$ over $\w$. In particular, we have that
	$$\displaystyle\int_\Omega\phi(|\nabla u|)\nabla u\nabla vdx=\int_\Omega fvdx,~v\in\w.$$
\end{prop}
\begin{proof}

Let $u\in\w$ be a fixed minimizer given in Proposition \ref{prop_min}. Consider $0<t<1$ and the function $v\in\w$. Now we mention that
$$
\displaystyle\frac{I(u+tv)-I(u)}{t}\geq 0.
$$
In other words, we have been ensured that
\begin{equation}\label{derivada_phi}
\displaystyle \int_\Omega\frac{\Phi(|\nabla u+t\nabla v|)-\Phi(|\nabla u|)}{t}dx\geq\int_\Omega fvdx.
\end{equation}
Now we claim that
$$
\displaystyle \lim_{t\rightarrow 0_+}\int_\Omega\frac{\Phi(|\nabla u+t\nabla v|)-\Phi(|\nabla u|)}{t}dx=\int_\Omega\phi(|\nabla u|)\nabla u
\nabla vdx.
$$
Now we shall prove the claim just above. Using the mean value theorem we deduce that
\begin{equation}\label{el1}
\Phi(|\nabla u+t\nabla v|)-\Phi(|\nabla u|)=\phi(\theta_t)\theta_t\left[|\nabla u+t\nabla v|-|\nabla u|\right],
\end{equation}
where $\theta: \Omega \rightarrow \mathbb{R}, \theta(x) \in [0,1], x \in \Omega$ and
\begin{equation}\label{el2}
\begin{array}{lll}
-|\nabla u|-|\nabla v| & \leq & \min\{|\nabla u+t\nabla v|,|\nabla u|\} \\ \\
& \leq & \theta_t\\ \\
& \leq & \max\{|\nabla u+t\nabla v|,|\nabla u|\}\\ \\
& \leq & |\nabla u|+|\nabla v| \mbox{ q.t.p. em }\Omega.
\end{array}
\end{equation}
According to \eqref{el2} we observe that $\theta_t(x)\rightarrow |\nabla u(x)|$ as $t\rightarrow 0$  a.e. $x \in \Omega$. This fact together with $(\ref{el1})$ imply that
\begin{equation}\label{el3}
\displaystyle\lim_{t\rightarrow 0}\frac{\Phi(|\nabla u+t\nabla v|)-\Phi(|\nabla u|)}{t}=\phi(|\nabla u|)\nabla u\nabla v~ \,\,\mbox{a.e in}
x \, \in \,\Omega.
\end{equation}
In addition, we also mention that
\begin{equation}\label{a_1}
\begin{array}{lll}
\displaystyle\frac{\Phi(|\nabla u+t\nabla v|)-\Phi(|\nabla u|)}{t} &   =  & \displaystyle\phi(\theta_t)\theta_t\left[\frac{|\nabla
	u+t\nabla v|-|\nabla u|}{t}\right]\\
\\
& \leq &  \displaystyle\phi(\theta_t)\theta_t\left[\frac{|\nabla u|+t|\nabla v|-|\nabla u|}{t}\right]\\
\\
& =    &  \displaystyle\phi(\theta_t)\theta_t|\nabla v|\\
\\
& \leq & \phi(|\nabla u|+|\nabla v|)(|\nabla u|+|\nabla v|)|\nabla v|.
\end{array}
\end{equation}
Here we have used the fact that $t\phi(t)\geq 0$ and $ t \rightarrow t\phi(t)$ is increasing for $t\geq 0$. Furthermore, we know that
\begin{equation}\label{a_2}
\begin{array}{lll}
\displaystyle\frac{\Phi(|\nabla u+t\nabla v|)-\Phi(|\nabla u|)}{t} & \geq & \displaystyle\phi(\theta_t)\theta_t\left[\frac{|\nabla
	u|-t|\nabla v|-|\nabla u|}{t}\right]\\
\\
&   =  & \displaystyle-\phi(\theta_t)\theta_t|\nabla v|\\
\\
& \geq & -\phi(|\nabla u|+|\nabla v|)(|\nabla u|+|\nabla v|)|\nabla v|.
\end{array}
\end{equation}
As a consequence estimates $(\ref{a_1})$ and $(\ref{a_2})$ imply that
\begin{equation}\label{el4}
\left|\displaystyle\frac{\Phi(|\nabla u+t\nabla v|)-\Phi(|\nabla u|)}{t}\right|\leq\phi(|\nabla u|+|\nabla v|)(|\nabla u|+|\nabla v|)|\nabla
v|.
\end{equation}
Moreover, using the estimate $\widetilde{\Phi}(t\phi(t))\leq \Phi(2t)$ for any $t\in\R$, we mention that
$$\phi(|\nabla u|+|\nabla v|)(|\nabla u|+|\nabla v|)\in L_{\widetilde{\Phi}}(\Omega).$$
Now, due H\"older's inequality we see that
\begin{equation*}
\phi(|\nabla u|+|\nabla v|)(|\nabla u|+|\nabla v|) |\nabla v|\in L^1(\Omega).
\end{equation*}
At this moment using $(\ref{el3})$, $(\ref{el4})$ and Dominated convergence Theorem we deduce that
$$\displaystyle\lim_{t\rightarrow 0_+}\int_\Omega\frac{\Phi(|\nabla u+t\nabla v|)-\Phi(|\nabla u|)}{t}dx=\int_\Omega\phi(|\nabla u|)\nabla u
\nabla vdx.$$
As a consequence, taking the limit as $t \rightarrow 0$ in $(\ref{derivada_phi})$ we observe that
\begin{equation*}
\int_\Omega \phi(|\nabla u|) \nabla u \nabla v dx \geq \int_\Omega f v dx, v \in \w.
\end{equation*}
In the last estimate changing $v$ by $- v$ we also obtain that
\begin{equation*}
\int_\Omega \phi(|\nabla u|) \nabla u \nabla v dx = \int_\Omega f v dx, v \in \w.
\end{equation*}
The proof of this proposition is now complete.
\end{proof}

In what follows we denote the inner product in $\mathbb{R}^{N}$ by $( , )$. At this moment we would like to show that problem \eqref{p1} admits exactly one solution in $\w$ for the reflexive case. In order to achieve this purpose we shall consider some auxiliary results listed just below.

\begin{prop}\label{mono}
	Let $\phi:(0,\infty)\rightarrow(0,\infty)$ be a fixed N-function satisfying hypotheses $(\phi_1)$ and $(\phi_2)$. Then
\begin{equation*}
	(\phi(|x|)x-\phi(|y|)y,x-y)\geq 0,~\forall~x,y\in{\R}^n
	\end{equation*}
and
	\begin{equation*}
	(\phi(|x|)x-\phi(|y|)y,x-y)>0,~\forall~x,y\in{\R}^n,~x\neq y.
	\end{equation*}
\end{prop}
\begin{proof}
Firstly, we shall split the proof into three parts. In the first one we put $x,y\in\R^N$ in such way that $|x| =|y|$. In this case we easily see that
\begin{equation*}
(\phi(|x|)x-\phi(|y|)y,x-y) = \phi(|x|) |x - y|^{2} \geq 0, x,y \in \mathbb{R}^{N}, |x| = |y|.
\end{equation*}
This estimate proves the proposition in the first part. In the second part we shall consider $x, y \in \mathbb{R}^{N}$ in such way that $|x|<|y|$.
Thanks to hypothesis $(\phi_{2})$ we mention that
	\begin{equation*}\label{rela-monotonica}
	\begin{array}{lll}
	(\phi(|x|)x-\phi(|y|)y,x-y) & \geq & \phi(|x|)|x|(|x|-|y|)+\phi(|y|)|y|(|y|-|x|)\\
	& = & (\phi(|x|)|x|-\phi(|y|)|y|)(|x|-|y|)> 0
	\end{array}
	\end{equation*}
This ends the proof in the second case. In the last part we shall consider $x, y \in \mathbb{R}^{N}$ in such way that $|x| > |y|$. Using the same ideas discussed just above we conclude one more time that
\begin{equation*}\label{rela-monotonica2}
	\begin{array}{lll}
	(\phi(|x|)x-\phi(|y|)y,x-y) & \geq & \phi(|x|)|x|(|x|-|y|)+\phi(|y|)|y|(|y|-|x|)\\
	& = & (\phi(|x|)|x|-\phi(|y|)|y|)(|x|-|y|)> 0.
	\end{array}
	\end{equation*}
This finishes the proof for this proposition.
\end{proof}

\begin{prop}\label{L_mono}
	Suppose $(\phi_1)-(\phi_2)$. Then we have that
	$$\int_{\Omega} (\phi(|\nabla u|)\nabla u-\phi(|\nabla v|)\nabla v)(\nabla u-\nabla v )dx> 0,~u,v\in\w,~u\neq v.$$
\end{prop}
\begin{proof}
Let $u,v\in\w$ be fixed functions in such way that $u \neq v$. Using Proposition \ref{mono} we deduce that
$$(\phi(|\nabla u|)\nabla u-\phi(|\nabla v|)\nabla v)(\nabla u-\nabla v )\geq 0~~\mbox{a.e in}~~\Omega.$$
Using the fact that $u\neq v$ there exits $\Omega_0\subset\Omega$ with positive Lebesgue measure such that
$$(\phi(|\nabla u|)\nabla u-\phi(|\nabla v|)\nabla v)(\nabla u-\nabla v )> 0~~\mbox{a.e in}~~\Omega_0.$$
As a consequence we obtain
$$\int_{\Omega_0} (\phi(|\nabla u|)\nabla u-\phi(|\nabla v|)\nabla v)(\nabla u-\nabla v )dx > 0.$$
The last estimate implies that
$$\int_{\Omega} (\phi(|\nabla u|)\nabla u-\phi(|\nabla v|)\nabla v)(\nabla u-\nabla v )dx > 0.$$
This ends the proof for this proposition.
\end{proof}

\begin{prop}\label{unic}
	Suppose $(\phi_1)-(\phi_2)$. Then the problem \eqref{p1} admits at most one solution in $\w$.
\end{prop}
\begin{proof}
The proof for this proposition follows immediately from Proposition \ref{L_mono}. We omit the details.
 \end{proof}

\begin{theorem}\label{thaux}
	Suppose $(\phi_1)-(\phi_2),~(\phi_3)^\prime$ and $\ell>1$. Then problem \eqref{p1} admits at least exactly one solution $u \in \w$.
\end{theorem}
\begin{proof}
The proof follows using Proposition \ref{teo_prin} and Proposition \ref{unic}.
\end{proof}

\section{Problem \eqref{p1} for the nonreflexive case}
In this section we shall prove some useful results in order to ensure existence of solution for problem \eqref{p1} in the nonreflexive case.
The first result in this direction is to consider a sequence of approximated quasilinear elliptic problems given by
\begin{equation}\label{aux}
\left\{\
\begin{array}{l}
\displaystyle-\Delta_\Phi{_{\epsilon}} u = f(x),~\mbox{in}~\Omega,\\ \\
u=0~\mbox{on}~\partial \Omega
\end{array}
\right.
\end{equation}
where $f \in L^{N}(\Omega)$ and $\Phi_\epsilon(t):=\Phi(t)+\frac \epsilon m t^m,~t\in\mathbb{R}, \epsilon > 0$. It is important to remember  that
$v \in W^{1,\Phi_{\epsilon}}_{0}(\Omega)$ is a weak solution for the problem \eqref{aux} when
\begin{equation}\label{test}
\displaystyle \epsilon \int_{\Omega} |\nabla v|^{m-2} \nabla v \nabla w dx + \int_{\Omega} \phi(|\nabla v|) \nabla v \nabla w dx = \int_{\Omega} f w dx
\end{equation}
holds true for any  $w \in  W^{1,\Phi_{\epsilon}}_{0}(\Omega)$. The main properties for the functions $\Phi_\epsilon$ can be read as

\begin{lem}\label{ell-epsilon}
	Suppose $(\phi_1)-(\phi_4)$. Then the function $\Phi_\epsilon(t)$ satisfies the following properties
	\begin{itemize}
		\item[(i)] $\Phi_\epsilon\rightarrow \Phi$ as $\epsilon\rightarrow0$;
		\item[(ii)] $\displaystyle 1 < \ell_\epsilon \leq \frac{\Phi_\epsilon^\prime(t)t}{\Phi_\epsilon(t)}\leq m,~t>0$;
		\item[(iii)] $\ell_\epsilon\rightarrow1$ as $\epsilon\rightarrow0$;
		\item[(iv)] $\Phi_\epsilon$ is equivalent to the N-function $t^m$.
	\end{itemize}
\end{lem}
\begin{proof}
First of all, we mention that $\Phi$ and $t\mapsto \frac\epsilon mt^m$ are N-function. Hence $\Phi_\epsilon$ is also a N-function. Moreover, we easily see that the limit in $(i)$ is verified.

At this moment we shall prove the item $(ii)$. Taking into account $(\phi_3)$ we have that
\begin{equation}\label{Phi}
	1\leq \frac{\phi(t)t^2}{\Phi(t)}\leq m, t > 0.
\end{equation}
As a consequence, we infer that
\begin{equation*}
\displaystyle\frac{\Phi^\prime_\epsilon(t)t}{\Phi_\epsilon(t)} = \frac{\epsilon t^m+\phi(t)t^2}{\Phi(t)+\frac\epsilon m t^m} = \frac{\epsilon \frac{t^m}{\Phi(t)}+\frac{\phi(t)t^2}{\Phi(t)}}{1+\frac \epsilon m \frac{t^m}{\Phi(t)}} \leq  \frac{\epsilon \frac{t^m}{\Phi(t)}+m}{1+\frac\epsilon m \frac{t^m}{\Phi(t)}}=m.
\end{equation*}
On the other hand, using one more time \eqref{Phi}, we also have that
\begin{eqnarray}\label{est-inf}
	\frac{\Phi^\prime_\epsilon(t)t}{\Phi_\epsilon(t)} = \frac{\epsilon \frac{t^m}{\Phi(t)}+\frac{\phi(t)t^2}{\Phi(t)}}{1+\frac\epsilon m \frac{t^m}{\Phi(t)}}\geq \frac{\epsilon \frac{t^m}{\Phi(t)}+1}{1+\frac\epsilon m \frac{t^m}{\Phi(t)}}=h\left(\frac{t^m}{\Phi(t)}\right),
\end{eqnarray}
where we define
$$h(s):=\frac{\epsilon s+1}{\frac\epsilon m s +1}.$$
It is easy to see that $h$ is increasing. Furthermore, we observe that
$$\frac{d}{dt}\left(\frac{t^m}{\Phi(t)}\right)=\frac{t^{m-1}}{\Phi(t)}\left(m-\frac{\phi(t)t^2}{\Phi(t)}\right)\geq 0.$$
As a product we obtain $\displaystyle t\mapsto \frac{t^m}{\Phi(t)}$ is nondecreasing. Hence the function $t\mapsto h\left(\frac{t^m}{\Phi(t)}\right)$ is also nondecreasing. As a consequence, using the estimate \eqref{est-inf}, we have been proven that
$$\frac{\Phi^\prime_\epsilon(t)t}{\Phi_\epsilon(t)}\geq\lim_{t\rightarrow 0}h\left(\frac{t^m}{\Phi(t)}\right)=\frac{\epsilon a+1}{\frac \epsilon m a+1}=1+\frac{(m-1)\epsilon a}{\epsilon a+m}=:\ell_\epsilon>1.$$
Here we have used the fact that $\Phi$ is a N-function showing that $m>1$. According to the last estimate we see that
$$\lim_{\epsilon\rightarrow 0+}\ell_\epsilon= 1.$$
This ends the proof for the item $(iii).$ Moreover, using Proposition \ref{lema_naru}, we infer that
$$\frac \epsilon m t^m\leq \Phi_\epsilon(t)\leq (\Phi(1)+\frac\epsilon m)t^m,~t\geq 1.$$
So that the proof of item $iv)$ is now achieved. This finishes the proof of this proposition.
\end{proof}

In what follows we shall consider the approximated elliptic problem \eqref{aux} that admits exactly one solution $u_{\epsilon}$ for any $\epsilon > 0$. This fact is verified thanks the inequality $\ell_{\epsilon}  > 1$ which implies that $W^{1,\Phi_{\epsilon}}_{0}(\Omega)$ is a reflexive Banach space, see Proposition \ref{teo_prin} and Proposition \ref{unic} and Theorem \ref{thaux}. In the nonreflexive case where $\ell = 1$ we shall consider some powerful results in order to get a weak solution for the problem \eqref{p1}. Initially, we shall use the approximate problem \eqref{aux}
obtained a sequence of bounded weak solutions $(u_{\epsilon}) \in W^{1, \Phi_{\epsilon}}_{0}(\Omega)$. So that we shall consider the following result

\begin{prop}\label{pr1}
Suppose $(\phi_{1})- (\phi_{3})$ where $\ell = 1$. Then the sequence $(u_{\epsilon})\subset L^{\infty}(\Omega)$ is bounded in $\w$ and $W^{1,1}_{0}(\Omega)$.
\end{prop}
\begin{proof}
Let $u_{\epsilon}$ be the unique solution for the auxiliary problem \eqref{aux} which is given by Theorem \ref{thaux}. Note that $u_{\epsilon}$ is in $\w \cap L^{\infty}(\Omega)$ for any $\epsilon > 0$, see Theorem \ref{taux}. Now, using the results discussed in Section 2 for Orlicz and Orlicz-Sobolev spaces, we mention that the following embedding are continuous  $W^{1,m}_{0}(\Omega) \hookrightarrow W^{1, \Phi_{\epsilon}}_{0}(\Omega) \hookrightarrow W^{1, \ell_{\epsilon}}_{0}(\Omega)$ (cf. \cite{clement}) and $W^{1, \Phi_{\epsilon}}_{0}(\Omega) \hookrightarrow W^{1,1}_{0}(\Omega)$ (cf. \cite{A}).

On the other hand, we observe that
\begin{equation*}
\Phi(t), \frac\epsilon m t^{m} \leq \Phi_{\epsilon}(t), t \geq 0.
\end{equation*}
As a product $\Phi , \frac \epsilon m t^{m} \prec \Phi_{\epsilon}$ proving that $L^{\Phi_{\epsilon}}(\Omega) \hookrightarrow L^{\Phi}(\Omega)$  and
$L^{\Phi_{\epsilon}}(\Omega) \hookrightarrow L^{m}(\Omega)$. Furthermore, we infer that
$W^{1,\Phi_{\epsilon}}_{0}(\Omega) \hookrightarrow \w$ and  $W^{1,\Phi_{\epsilon}}_{0}(\Omega) \hookrightarrow W^{1,m}_{0}(\Omega)$.
As a consequence the last embedding says also that $W^{1,\Phi_{\epsilon}}(\Omega) = W^{1,m}_{0}(\Omega)$. In particular, we obtain that
$u_{\epsilon} \in W^{1,m}_{0}(\Omega)$ for any $\epsilon > 0$.

Now we shall prove that $u_{\epsilon}$ is bounded in $W^{1,\Phi}_{0}(\Omega)$. Putting $u_{\epsilon}$ as testing function in \eqref{test} we easily see that
\begin{equation*}
\epsilon \int_{\Omega} |\nabla u_{\epsilon}|^{m} dx + \int_{\Omega} \phi(|\nabla u_{\epsilon}|) |\nabla u_{\epsilon}|^{2} dx = \int_{\Omega} f u_{\epsilon} dx.
\end{equation*}
Using Holder's inequality we also see that
\begin{equation*}
\epsilon \int_{\Omega} |\nabla u_{\epsilon}|^{m} dx + \int_{\Omega} \phi(|\nabla u_{\epsilon}|) |\nabla u_{\epsilon}|^{2} dx \leq \|f\|_{N} \|u_{\epsilon}\|_{1^{\star}}.
\end{equation*}
Taking into account the embedding $W^{1,1}_{0}(\Omega) \hookrightarrow L^{1^{\star}}(\Omega)$ there exists $S = S(N,\Omega) > 0$ in such way that
\begin{equation*}
\|v\|_{1^{\star}} \leq S \|v\|_{W^{1,1}_{0}(\Omega)}, v \in W^{1,1}_{0}(\Omega).
\end{equation*}
As a consequence the last embedding and hypothesis $(\phi_{3})$ imply that
\begin{eqnarray}\label{e0}
\int_{\Omega} \Phi(|\nabla u_{\epsilon}|) dx &\leq& \int_{\Omega} \phi(|\nabla u_{\epsilon}|) |\nabla u_{\epsilon}|^{2} dx \nonumber \\
&\leq& m\epsilon \int_{\Omega} |\nabla u_{\epsilon}|^{m} dx + \int_{\Omega} \phi(|\nabla u_{\epsilon}|) |\nabla u_{\epsilon}|^{2} dx \nonumber \\
&\leq& \|f\|_{N} \|u_{\epsilon}\|_{1^{\star}} \leq S \|f\|_{N} \|u_{\epsilon}\|_{W^{1,1}_{0}(\Omega)}.
\end{eqnarray}
Let $K > 0$ be fixed. Using the last estimate and hypothesis $(\phi_{2})$ it follows that
\begin{eqnarray}
\|u_{\epsilon}\|_{W^{1,1}_{0}(\Omega)} &=& \int_{|\nabla u_{\epsilon}| \leq K} |\nabla u_{\epsilon}| dx + \int_{|\nabla u_{\epsilon}| > K} |\nabla u_{\epsilon}|dx \nonumber \\
&\leq& K |\Omega|  + \dfrac{1}{K \phi(K)}\int_{|\nabla u_{\epsilon}| > K} \phi(|\nabla u_{\epsilon}|)|\nabla u_{\epsilon}|^{2}dx. \nonumber
\end{eqnarray}
Putting the all estimates just above together we obtain
\begin{equation}\label{e1}
\|u_{\epsilon}\|_{W^{1,1}_{0}(\Omega)} \leq K |\Omega| + \dfrac{S \|f\|_{N}}{K \phi(K)} \|u_{\epsilon}\|_{W^{1,1}_{0}(\Omega)}.
\end{equation}
Now due the fact that $\displaystyle \lim_{K \rightarrow \infty} K \phi(K) = \infty$ there exists $K_{0} > 0$ such that
\begin{equation*}
\dfrac{S \|f\|_{N}}{K \phi(K)} < 1
\end{equation*}
for any $K \geq K_{0}$. In particular, using inequality \eqref{e1}, we infer that
\begin{equation*}
\|u_{\epsilon}\|_{W^{1,1}_{0}(\Omega)} \leq \dfrac{ K |\Omega|}{1 - \frac{S \|f\|_{N}}{K \phi(K)}}.
\end{equation*}
Furthermore, taking into account hypothesis $(\phi_{3})$ and \eqref{e0}, we also mention that
\begin{equation*}
\int_{\Omega} \Phi(|\nabla u_{\epsilon}|) dx \leq \int_{\Omega} \phi(|\nabla u_{\epsilon}|) |\nabla u_{\epsilon}|^{2} dx \leq \dfrac{S \|f\|_{N} K|\Omega|}{1 - \frac{S \|f\|_{N}}{K \phi(K)}}.
\end{equation*}
Now we define
\begin{equation*}
R = \max \left\{ \dfrac{S \|f\|_{N} K|\Omega|}{1 - \frac{S \|f\|_{N}}{K \phi(K)}}, \dfrac{K|\Omega|}{1 - \frac{S \|f\|_{N}}{K \phi(K)}}  \right\}.
\end{equation*}
As a consequence we have been shown that
\begin{equation}\label{e3}
\int_{\Omega} \Phi(|\nabla u_{\epsilon}|) dx \leq R, \int_{\Omega} \phi(|\nabla u_{\epsilon}|)|\nabla u_{\epsilon}|^{2} dx \leq R, \int_{\Omega} |\nabla u_{\epsilon}| dx \leq R.
\end{equation}
According to Lemma \ref{lema_naru} in the Appendix  it follows that
\begin{equation*}
\min \left( \|u_{\epsilon}\|_{W^{1,\Phi}_{0}(\Omega)}, \|u_{\epsilon}\|^{m}_{W^{1,\Phi}_{0}(\Omega)} \right) \leq \int_{\Omega} \Phi(|\nabla u_{\epsilon}|) dx \leq R.
\end{equation*}
Hence the sequence $(u_{\epsilon})$ is now bounded in $\w$ and $W^{1,1}_{0}(\Omega)$. This completes the proof.
\end{proof}

\begin{prop}\label{S+}
	Suppose $(\phi_1),(\phi_2),(\phi_3)'$. Let $(u_{n}) \in \w$ be a sequence satisfying
	\begin{description}
	  \item[i)] $u_{n} \stackrel{*}\rightharpoonup u$ in $\w$;
	  \item[ii)] $\displaystyle \limsup_{n \rightarrow \infty} \langle -\Delta_\Phi u_{n}, u_{n} - u \rangle \leq 0$.
	\end{description}
Then we obtain that $u_{n} \rightarrow u$ in $\w$. Under this condition we say that the operator $\Phi$-Laplacian is of $(S)^{+}$ type.
\end{prop}
\begin{proof}
The proof is similar to the proof of \cite[Prop. 3.5]{Correa} replacing the weak convergence $u_{n}\rightharpoonup u$ by the weak star convergence $u_{n} \stackrel{*}\rightharpoonup u$. For the reader convenience we give a sketch for the proof.  Here we emphasize one more time that $\w$ is not reflexive anymore. However, the Orlicz-Sobolev space $\w$ is isomorphic to a closed set in the weak star topology. More precisely, we mention that
$$\w \subseteq \displaystyle \prod_{j=1}^{N + 1} L^{\Phi}(\Omega) \simeq \left( \prod_{j = 1}^{N + 1} E_{\Phi}\right)^{\star}$$
where $E_{\Phi}$ is a separable space. Under these conditions the proof following the same ideas discussed in \cite[Prop. 3.5]{Correa}. This ends the proof.
\end{proof}

\begin{prop}\label{pr2}
Suppose $(\phi_{1})- (\phi_{4})$ where $\ell = 1$. Then the problem \eqref{p1} admits at least one solution $u \in \w $.
\end{prop}
\begin{proof}
Let $u_{\epsilon} \in W^{1,\Phi_{\epsilon}}_{0}(\Omega)$ be the unique solution for the auxiliary elliptic problem \eqref{aux}. According to Proposition \ref{pr1} we infer that $u_{\epsilon}$ is bounded in $\w$ and $W^{1,1}_{0}(\Omega)$. As a consequence $u_{\epsilon} \stackrel{*}\rightharpoonup u$
in the weak star topology. Indeed, the Orlicz-Sobolev space $\w$ is isomorphic to a closed set in the weak star topology. More precisely, as was mentioned before we observe that
$$\w \subseteq \displaystyle \prod_{j=1}^{N + 1} L^{\Phi}(\Omega) \simeq \left( \prod_{j = 1}^{N + 1} E_{\Phi}\right)^{\star}$$
where $E_{\Phi}$ is a separable space. For further results on weak star topologies we refer the reader to Gossez \cite{gossez-Czech,Gz1}.

Now, using the weak star converge for $u_{\epsilon}$, we observe that
\begin{eqnarray}
\int_{\Omega} |\nabla u | |\nabla u_{\epsilon}| \phi(|\nabla u_{\epsilon}|) dx \leq C\nonumber
\end{eqnarray}
holds true for some $C > 0$. In fact, using Young's inequality and the $\Delta_{2}$ condition for $\Phi$, we have that
\begin{eqnarray}
 |\nabla u | |\nabla u_{\epsilon}| \phi(|\nabla u_{\epsilon}|) &\leq& \Phi(|\nabla u|) + \tilde{\Phi}(|\nabla u_{\epsilon}| \phi(|\nabla u_{\epsilon}|)) \nonumber \\
&\leq&  \Phi(|\nabla u|) + \Phi( 2 |\nabla u_{\epsilon}|) \leq  \Phi(|\nabla u|) + 2^{m} \Phi(|\nabla u_{\epsilon}|) \nonumber
\end{eqnarray}
Hence the last estimate together with \eqref{e3} imply that
\begin{equation*}
\int_{\Omega}  |\nabla u | |\nabla u_{\epsilon}| \phi(|\nabla u_{\epsilon}|) dx \leq  \int_{\Omega} [\Phi(|\nabla u|) + 2^{m} \Phi(|\nabla u_{\epsilon}|)] dx \leq R + 2^{m}R.
\end{equation*}

Now we claim that u is a weak solution to the elliptic problem \eqref{p1}. Here we note that $u$ is not in general a testing function for the auxiliary elliptic problem \eqref{aux}. In this way, we shall consider a density argument in order to prove the claim just above. More specifically,
we know that $C^{\infty}_{0}(\Omega)$ is dense in $W^{1,1}_{0}(\Omega)$ and $\w$. As a product there exists a sequence $(U_{k})$ in $C^{\infty}_{0}(\Omega)$ in such way that
\begin{equation}\label{e4}
\|u - U_{k}\|_{W^{1,1}_{0}(\Omega)}, \|u - U_{k}\|_{\w}  \leq \dfrac{1}{k}.
\end{equation}
Using $u_{\epsilon} - U_{k}$ as testing function in the problem \eqref{aux} we mention that
\begin{equation}\label{e8}
 \epsilon \langle -\Delta_{m} u_{\epsilon} , u_{\epsilon} - U_{k} \rangle + \langle - \Delta_{\Phi_{\epsilon}}u_\epsilon
  , u_{\epsilon} - U_{k} \rangle = \int_{\Omega} f (u_{\epsilon} - U_{k}) dx.
\end{equation}
The last identity says also that
\begin{equation*}
- \epsilon \int_{\Omega} |\nabla u_{\epsilon}|^{m-2} \nabla u_{\epsilon} \nabla U_{k} dx  + \int_{\Omega} \phi(|\nabla u_{\epsilon}|) \nabla u_{\epsilon} \nabla ( u_{\epsilon} - U_{k}) dx  \leq \int_{\Omega} f (u_{\epsilon} - U_{k}) dx.
\end{equation*}
The last inequality can be written in the following form
\begin{eqnarray}
- \epsilon \int_{\Omega} |\nabla u_{\epsilon}|^{m-2} \nabla u_{\epsilon} \nabla U_{k} dx &+& \int_{\Omega} \phi(|\nabla u_{\epsilon}|) \nabla u_{\epsilon} \nabla ( u_{\epsilon} - u) dx \nonumber \\
&+& \int_{\Omega} \phi(|\nabla u_{\epsilon}|) \nabla u_{\epsilon} \nabla (u - U_{k})\nonumber \\
&\leq& \int_{\Omega} f (u_{\epsilon} - u) dx + \int_{\Omega} f (u - U_{k}) dx. \nonumber
\end{eqnarray}
Moreover, we mention that $\phi(|\nabla u|)|\nabla u| |\nabla (u_{\epsilon} - u)| \in L^{1}(\Omega)$.

At this moment we claim that
\begin{equation*}
\left|\int_{\Omega} |\nabla u_{\epsilon}|^{m-2} \nabla u_{\epsilon} \nabla U_{k} dx \right| \leq C
\end{equation*}
holds for some $C > 0$ independent on $\epsilon > 0$. Indeed, the continuous embedding $W^{1,\Phi_{\epsilon}}(\Omega) \hookrightarrow W^{1,m}_{0}(\Omega)$
provide a positive number $C > 0$ in such way that
\begin{equation*}
\|v\|_{m} \leq C \|v\|_{W^{1,\Phi_{\epsilon}}_{0}(\Omega)}, v \in W^{1,\Phi_{\epsilon}}_{0}(\Omega).
\end{equation*}
Taking $v = u_{\epsilon}$ in the previous estimate we obtain
\begin{equation*}
\|u_{\epsilon}\|_{W^{1,m}_{0}(\Omega)} \leq C \|u_{\epsilon}\|_{\w} \leq C.
\end{equation*}
In other words, we have been shown that $(u_{\epsilon})$ is bounded in $W^{1,m}_{0}(\Omega)$ for any $\epsilon > 0$. Hence, using Holder's inequality and the estimate just above, we deduce
\begin{eqnarray}
\left| \int_{\Omega} |\nabla u_{\epsilon}|^{m-2}\nabla u_{\epsilon} \nabla U_{k} dx \right| &\leq&
 \| |\nabla u_{\epsilon}|^{m-1}\|_{m/m-1} \|U_{k}\|_{W^{1,m}_{0}(\Omega)} \nonumber \\
 &\leq& \|u_{\epsilon}\|_{W^{1,m}_{0}(\Omega)}^{m-1} \|U_{k}\|_{W^{1,m}_{0}(\Omega)} \leq C \|U_{k}\|_{W^{1,m}_{0}(\Omega)}. \nonumber
\end{eqnarray}
As a product taking the limit in the last inequality we see that
\begin{equation}\label{e6}
\lim_{\epsilon \rightarrow 0} \epsilon \int_{\Omega}  |\nabla u_{\epsilon}|^{m-2}\nabla u_{\epsilon} \nabla U_{k} dx = 0.
\end{equation}
On the other hand, due the weak star convergence, we also see that
\begin{equation}\label{e7}
\lim_{\epsilon \rightarrow 0} \int_{\Omega}  f ( u_{\epsilon} - u ) dx = 0.
\end{equation}
Now, using one more time the Holder's inequality,  we observe that
\begin{equation}\label{e9}
\lim_{k \rightarrow \infty} \int_{\Omega} f (u - U_{k}) dx = 0.
\end{equation}
In fact, using Orlicz-Sobolev embedding  and \eqref{e4}, we easily see that
\begin{equation*}
\left| \int_{\Omega} f (u - U_{k}) dx \right| \leq \|f\|_{N} \|u - U_{k}\|_{1^{\star}} \leq C \|f\|_{N}  \|u - U_{k}\| \leq \dfrac{C \|f\|_{N}}{k}.
\end{equation*}
as $k \rightarrow \infty$. Additionally, we claim also that
\begin{equation*}
\lim_{k \rightarrow \infty} \int_{\Omega} \phi (|\nabla u_{\epsilon}|) \nabla u_{\epsilon} \nabla (u - U_{k})dx = 0.
\end{equation*}
The proof for this claim follows the following ideas. Firstly, we shall use one more time Holder's inequality proving that
\begin{eqnarray}\label{e5}
\left| \int_{\Omega} \phi (|\nabla u_{\epsilon}|) \nabla u_{\epsilon} \nabla (u - U_{k})dx\right| &\leq& 2 \|\nabla (u - U_{k})\|_{\Phi} \|\phi(|\nabla u_{\epsilon}|) |\nabla u_{\epsilon}|\|_{\tilde{\Phi}}\nonumber\\
& \leq & \dfrac{2}{k} \|\phi(|\nabla u_{\epsilon}|) |\nabla u_{\epsilon}|\|_{\tilde{\Phi}}.
\end{eqnarray}
On the other hand, due the $\Delta_{2}$ condition for $\Phi$ and estimate \eqref{e3}, we get
\begin{equation*}
\int_{\Omega} \tilde{\Phi}(\phi(|\nabla u_{\epsilon}|) |\nabla u_{\epsilon}| ) \leq \int_{\Omega} \Phi(2|\nabla u_{\epsilon}|) dx \leq 2^{m} \int_{\Omega} \Phi(|\nabla u_{\epsilon}|) \leq 2^{m} R.
\end{equation*}
Now, using one more time that $\Phi$ is convex, we deduce that
\begin{equation*}
\|\phi(|\nabla u_{\epsilon}|) |\nabla u_{\epsilon}|\|_{\tilde{\Phi}} \int_{\Omega} \tilde{\Phi} \left(\dfrac{\phi(|\nabla u_{\epsilon}|) |\nabla u_{\epsilon}|}{\|\phi(|\nabla u_{\epsilon}|) |\nabla u_{\epsilon}|\|_{\tilde{\Phi}}}\right) \leq \int_{\Omega} \tilde{\Phi} (\phi(|\nabla u_{\epsilon}||\nabla u_{\epsilon}|) dx \leq 2^{m} R
\end{equation*}
holds true whenever $\|\phi(|\nabla u_{\epsilon}|) |\nabla u_{\epsilon}|\|_{\tilde{\Phi}} \geq 1$. Hence the last estimate shows that
\begin{equation*}
\|\phi(|\nabla u_{\epsilon}|) |\nabla u_{\epsilon}|\|_{\tilde{\Phi}} \leq \max(1, 2^{m}R).
\end{equation*}

Now taking into account \eqref{e5} we obtain that
\begin{equation*}
\lim_{k \rightarrow \infty} \int_{\Omega} \phi(|\nabla u_{\epsilon}|) \nabla u_{\epsilon} \nabla (u - U_{k}) dx = 0.
\end{equation*}
At this moment using \eqref{e6}, \eqref{e7}, \eqref{e9} and taking the limits as $\epsilon \rightarrow 0$
and $k\rightarrow \infty$ in the inequality \eqref{e8} we get
\begin{equation*}
\limsup_{\epsilon \rightarrow 0} \int_{\Omega} \phi(|\nabla u_{\epsilon}|) \nabla u_{\epsilon} \nabla (u_{\epsilon} - u) dx =0
\end{equation*}
Summing up, due the $(S^{+})$ condition, for the $\Phi$-Laplacian operator, we have that $u_{n} \rightarrow $u in $W^{1,\Phi}_{0}(\Omega)$ (cf. Proposition \ref{S+}). In this way it follows from  Dal Masso et al \cite{maso} that
\begin{equation*}
\nabla u_{\epsilon} \rightarrow \nabla u \,\, \mbox{a. e. in} \,\, \Omega.
\end{equation*}
Moreover, there exists $h \in L^{1}(\Omega)$ in such way that
\begin{equation*}
\Phi(|\nabla (u_{\epsilon} - u)|) \leq h \,\, \mbox{a. e. in} \,\, \Omega.
\end{equation*}
The last estimate says that
\begin{equation*}
|\nabla (u_{\epsilon} - u)| \leq \Phi^{-1}(h) \,\, \mbox{a. e. in}  \,\,\Omega.
\end{equation*}
As a consequence
\begin{equation*}
|\nabla u_{\epsilon}| \leq |\nabla (u_{\epsilon} - u)| +  |\nabla u| \leq \Phi^{-1}(h) + |\nabla u|.
\end{equation*}
In particular, using one more time Young's inequality and $\Delta_{2}$ condition for $\Phi$, we have that
\begin{eqnarray}
\phi(|\nabla u_{\epsilon}|) |\nabla u_{\epsilon}| |\nabla v| &\leq& \Phi(|\nabla v|) + \tilde{\Phi}(\phi(|\nabla u_{\epsilon}|) |\nabla u_{\epsilon}|) \nonumber \\
&\leq& \Phi(|\nabla v|) + \Phi(2 |\nabla u_{\epsilon}|) \leq \Phi(|\nabla v|) + 2^{m} \Phi(|\nabla u_{\epsilon}|).
\nonumber
\end{eqnarray}
Now, using the last estimate and due the convexity of $\Phi$, we obtain
\begin{equation*}
\phi(|\nabla u_{\epsilon}|) |\nabla u_{\epsilon}| |\nabla v| \leq  \Phi(|\nabla v|) + 2^{m} \Phi(|\nabla u| + \Phi^{-1}(h))
\leq \Phi(|\nabla v|) + 2^{2m} [\Phi(|\nabla u|) + h]
\end{equation*}
As a consequence the Lebesgue convergence theorem implies that
\begin{equation*}
\lim_{\epsilon \rightarrow 0} \int_{\Omega} \phi(|\nabla u_{\epsilon}|) \nabla u_{\epsilon} \nabla v  dx = \int_{\Omega} \phi(|\nabla u|) \nabla u \nabla v dx, v \in \w.
\end{equation*}
Putting all estimates together and taking the limit as $\epsilon \rightarrow 0$ in the equation
\begin{equation*}
\epsilon \langle - \Delta_{m} u_{\epsilon}, v \rangle + \int_{\Omega} \phi((|\nabla u_{\epsilon}|) \nabla u_{\epsilon} \nabla v dx = \int_{\Omega} f v dx, v \in \w
\end{equation*}
we conclude that
\begin{equation*}
\int_{\Omega} \phi(|\nabla u|) \nabla u \nabla v dx =  \int_{\Omega} f v dx, v \in \w
\end{equation*}
To sum up, $u \in \w$ is a weak solution for the problem \eqref{p1}. Using the same ideas discussed in the proof of Proposition \ref{unic} we know that problem \eqref{p1} admits at most
one solution. Consequently, the problem \eqref{p1} admits exactly one solution for each $f \in L^{N}(\Omega)$. This ends the proof.
\end{proof}

\begin{prop}
Suppose $(\phi_{1})- (\phi_{4})$ where $\ell = 1$. Then the problem \eqref{p1} admits exactly one solution $u \in \w $.
\end{prop}
\begin{proof}
The proof follows using the same ideas discussed in the proof of Proposition \ref{unic}. The main point here is to ensure that
$\Phi$-Laplace operator is strictly monotonic. As a consequence problem \eqref{p1} admits at most one solution $u \in \w$.
Besides that, using Proposition \ref{teo_prin}, there exists at least one solution $u \in \w$ for the problem \eqref{p1}. We omit the details.
\end{proof}

In what follows we shall consider the elliptic problem \eqref{p2} under superlinear conditions. One more time we define the auxiliary elliptic problem
\begin{equation}\label{aux2}
\left\{\
\begin{array}{l}
\displaystyle- \Delta_{\Phi_{\epsilon}} u = g(x,u),~\mbox{in}~\Omega,\\ \\
u=0~\mbox{on}~\partial \Omega
\end{array}
\right.
\end{equation}
where $\epsilon > 0$ and $\Phi_{\epsilon}(t) =\frac\epsilon m t^{m} + \Phi(t), t \geq 0$. Here is important to recover the definition for weak solution $u \in W_{0}^{1,\Phi_{\epsilon}}(\Omega)$ to the problem \eqref{aux2} which is given by
\begin{equation*}
\int_{\Omega} \phi_{\epsilon}(|\nabla u|) \nabla u \nabla w dx = \int_{\Omega} g(x, u) w dx, w \in W_{0}^{1,\Phi_{\epsilon}}(\Omega).
\end{equation*}
Weak solution for this problem are precisely the critical point for the functional $J : W_{0}^{1,\Phi_{\epsilon}}(\Omega) \rightarrow \mathbb{R}$
given by
\begin{equation*}
J(u) = \int_{\Omega} \Phi_{\epsilon}(|\nabla u|) dx - \int_{\Omega} G(x, u)dx
\end{equation*}
where $G(x, t) = \int_{0}^{t} g(x,s)ds, t \in \mathbb{R}, x \in \Omega$. As a consequence finding weak solutions to the problem \eqref{p2} is equivalent to find critical points for $J$. Using the approximated problem \eqref{aux2} we observe that $J$ satifies the Cerami condition for any
$\epsilon > 0$, see Carvalho et al \cite{JVMLED}. In addition, using hypotheses $(\phi_{1})-(\phi_{3})$ and $(g_{1})-(g_{4})$, the functional $J$ possesses the mountain pass geometry, see Carvalho et al \cite{JVMLED}. In this way, we shall consider the following existence result

\begin{prop}
Suppose $(\phi_{1})- (\phi_{3})$ where $\ell = 1$. Assume also that $(g_{1})-(g_{4})$ holds. Then the problem \eqref{aux2} admits at least one weak solution in $u_{\epsilon} \in W_{0}^{1,\Phi_{\epsilon}}(\Omega)$ for each $\epsilon > 0$. Furthermore, using regularity results, we also mention that $u_{\epsilon}$ is in $C^{1,\alpha_\epsilon}(\overline\Omega)$, for some $\alpha_\epsilon>0$.
\end{prop}
\begin{proof}
First of all, we recall that $W_{0}^{1,\Phi_{\epsilon}}(\Omega)$ is Banach reflexive due the fact that $\ell_{\epsilon} > 1$ for each $\epsilon > 0$.
As a consequence, using the mountain pass theorem, we know that the Problem \eqref{p2} admits at least one solution $u_{\epsilon} \in W_{0}^{1,\Phi_{\epsilon}}(\Omega)\cap C^{1,\alpha_\epsilon}(\overline\Omega)$ for each $\epsilon > 0$, see Carvalho et al \cite{JVMLED}.  We omit the details.
\end{proof}

\begin{prop}\label{l0}
Suppose $(\phi_{1})- (\phi_{4})$ where $\ell = 1$. Assume also that $(g_{1})-(g_{4})$ holds. Then the problem \eqref{p2} admits at least one weak solution in $u\in \w$.
\end{prop}
\begin{proof}
The proof follows along to the same lines discussed in the proof of Proposition \ref{pr2}. Here we omit the proof.
\end{proof}

At this moment we shall consider the truncation functions given by
\begin{equation*}
g^{+}(x,t) = \left\{\
\begin{array}{l}
g(x, t), t \geq 0, x \in \Omega \\ \\
0, t < 0, x \in \Omega
\end{array}
\right.
\end{equation*}
and
\begin{equation*}
g^{-}(x,t) = \left\{\
\begin{array}{l}
g(x, t), t \leq 0, x \in \Omega \\ \\
0, t > 0, x \in \Omega.
\end{array}
\right.
\end{equation*}
At the same time we define the functionals $J^{\pm} : W^{1,\Phi_{\epsilon}}_{0}(\Omega) \rightarrow \mathbb{R}$ given by
\begin{equation*}
J^{\pm}(u) = \int_{\Omega} \Phi_{\epsilon}(|\nabla u|) dx - \int_{\Omega} G^{\pm}(x, u)dx
\end{equation*}
where $G^{\pm}(x, t) = \int_{0}^{t} g^{\pm}(x,s)ds, t \in \mathbb{R}, x \in \Omega$. It is not hard to verify that $J^{\pm}$
admits the mountain pass geometry. As a consequence we shall consider the following result

\begin{prop}\label{l1}
Suppose $(\phi_{1})- (\phi_{4})$ and $\ell = 1$ holds true. Assume also that $(g_{1})-(g_{4})$ holds. Then the problem \eqref{p2} admits at least two nontrivial weak solutions $u_{1}, u_{2} \in \w$.
\end{prop}
\begin{proof}
The proof follows using the Mountain Pass Theorems for the functionals $J^{\pm}$. One more time we mention that $J^{\pm}$
satisfies the Cerami condition for each $\epsilon > 0$, see Carvalho et al \cite{JVMLED}. In this way we obtain two sequences
$u^{+}_{\epsilon}, u^{-}_{\epsilon} \in W^{1,\Phi_{\epsilon}}_{0}(\Omega)$ of critical points for $J^{+}$ and $J^{-}$, respectively.

At this stage we claim that there exists $r_{0} > 0$ in such way that $J^{\pm}(u^\pm_\epsilon)\geq r_0$ where
$r_{0}$ does not depend on $\epsilon > 0$. In fact, using $(\psi_1)$ and $(g_4)$, given $0<\eta<\lambda_1$ there exist $C,\delta>0$ such that
	$$G^{\pm}(x,t)<(\lambda_1-\eta)\Phi(t)+C\Psi(t),~t\in \mathbb{R}.$$
Hence, taking into account Poincar\'e inequality and using the estimate $\Phi_\epsilon(t)\geq \Phi(t),~t\in\mathbb{R}$, $\w\hookrightarrow L_\Psi(\Omega)$ , we mention that
	\begin{eqnarray}
		J_\epsilon^{\pm}(u) &\geq& \int_\Omega \Phi_\epsilon(|\nabla u|)dx-(\lambda_1-\eta)\int_\Omega\Phi(u)dx-C\int_\Omega \Psi(u)dx\nonumber\\
		&\geq& \frac{\eta}{\lambda_1}\int_\Omega \Phi(|\nabla u|)dx-C\int_\Omega \Psi(u)dx\nonumber\\
		&\geq& \frac{\eta}{\lambda_1}\min\{\|u\|,\|u\|^m\}-C\max\{\|u\|^{\ell_\Psi},\|u\|^{m_\Psi}\}\nonumber\\
		&=&\|u\|^m\left(\frac{\eta}{\lambda_1}-C\|u\|^{\ell_{\Psi}-m}\right). \nonumber
	\end{eqnarray}
holds true for any $\|u\|\leq 1$ holds true for any $\|u\|\leq 1$.

Now using the same ideas discussed in the proof of Proposition \ref{pr2} we point out that $u_{\epsilon}^{+} \stackrel{*}\rightharpoonup u_{1}$ and $u_{\epsilon}^{-} \stackrel{*}\rightharpoonup u_{2}$ in the weak star topology. Furthermore, the functional $J^{\pm}$ is weak star lower semicontinuous. Now applying Proposition \ref{S+} we deduce that $u_{n}^{+} \rightarrow u_{1}$ and $u_{n}^{-} \rightarrow u_{2}$ in $\w$. Hence, taking the negative part of $u_{1}$ as testing function, we obtain that $u_{1} \geq 0$ in $\Omega$. Similarly, we also obtain $u_{2} \leq 0$ in $\Omega$. As a consequence $u_{1}, u_{2}$ are nontrivial critical points to the functional $J$ which give us weak nontrivial solutions to the elliptic problem \eqref{p2}.  This finishes the proof.
\end{proof}

\section{Regularity results on quasilinear elliptic problems}

In this section we prove a regularity result for the problem (\ref{p1}).

\begin{theorem}\label{taux} Assume that there exist positive constants $C_1,C_2$ such that \begin{equation}C_1\le \frac{\Phi(t)}{t^m}\le C_2,\ \forall\ t>0.\label{equim}\end{equation}

If $f\in L^q(\Omega)$ with $q>N/m$ and $u$ is a solution of $\eqref{p1}$ then, $u\in L^\infty(\Omega)$.
\end{theorem}

The proof uses Moser's iteration technique and goes as follows:

\begin{proof} For $R>0$, define $\Omega_R=\{x\in \Omega:\ \operatorname{dist}(x,\partial\Omega)>R\}$. For $0<R_2<R_1$, let $\varphi=\eta^m(\overline{u}_s^{m\alpha}\overline{u}-k^{m\alpha+1})$ where $\alpha$ is a parameter to be choosen conveniently later, $\overline{u}=\max\{u,k\}$ for $k>0$, $\overline{u}_s=\min\{\overline{u},s\}$ and $\eta\in C^1_0(\Omega)$ satisfies   $\eta=1$ in $\Omega_{R_1}$, $\eta=0$ in $\Omega\setminus \Omega _{R_2}$, $\eta \ge 0$ and $|\nabla \eta|\le C/(R_1-R_2)$ for some positive constant $C$.

Note that $\varphi\in W_0^{1,\Phi}(\Omega)$ and $\nabla \varphi=\eta^m[m\alpha \overline{u}_s^{m\alpha-1}\overline{u}\nabla \overline{u}_s+\overline{u}_s^{m\alpha}\nabla\overline{u}]+m\eta^{m-1}(\overline{u}_s^{m\alpha}\overline{u}-k^{m\alpha+1})\nabla \eta$. We substitute it in the equation (\ref{p1}) to find that

\begin{equation}\label{M1}\begin{split}& m\alpha\int \eta^m\overline{u}_s^{m\alpha-1}\overline{u}\phi(|\nabla u|)\nabla u\nabla \overline{u}_s+\int\eta^m\overline{u}_s^{m\alpha}\phi(|\nabla u|)\nabla u\nabla \overline{u} \\ &+ m\int \eta^{m-1}(\overline{u}_s^{m\alpha}\overline{u}-k^{m\alpha+1})\phi(|\nabla u|)\nabla u\nabla \eta\int f\eta^m(\overline{u}_s^{m\alpha}\overline{u}-k^{\alpha m+1}).\end{split}
\end{equation}

For $u\le k$ we have that $\nabla \overline{u},\nabla \overline{u}_s=0$ and $(\overline{u}_s^{m\alpha}\overline{u}-k^{m\alpha+1})=0$, therefore, from (\ref{M1}) we must conclude

\begin{equation}\label{M2}\begin{split}& m\alpha\int \eta^m\overline{u}_s^{m\alpha-1}\overline{u}\phi(|\nabla \overline{u}_s|)|\nabla \overline{u}_s|^2+\int\eta^m\overline{u}_s^{m\alpha}\phi(|\nabla \overline{u}|)|\nabla \overline{u}|^2 \\ &+ m\int \eta^{m-1}(\overline{u}_s^{m\alpha}\overline{u}-k^{m\alpha+1})\phi(|\nabla \overline{u}|)\nabla \overline{u}\nabla \eta=\int f\eta^m(\overline{u}_s^{m\alpha}\overline{u}-k^{\alpha m+1}).\end{split}\end{equation}

Note that $(\overline{u}_s^{m\alpha}\overline{u}-k^{m\alpha+1})\le \overline{u}_s^{m\alpha}\overline{u}$, which implies that $ m\eta^{m-1}(\overline{u}_s^{m\alpha}\overline{u}-k^{m\alpha+1})\phi(|\nabla \overline{u}|)\nabla \overline{u}\nabla \eta\le m\eta^{m-1}\overline{u}_s^{m\alpha}\overline{u}\phi(|\nabla \overline{u}|)|\nabla \overline{u}||\nabla \eta|$. From Young's inequality we obtain that \begin{equation}\label{M3}
 m\eta^{m-1}(\overline{u}_s^{m\alpha}\overline{u}-k^{m\alpha+1})\phi(|\nabla \overline{u}|)\nabla \overline{u}\nabla \eta\le m( \tilde{\Phi}(\epsilon\eta^{m-1}\phi(|\nabla \overline{u}||\nabla \overline{u}|))+\Phi(\overline{u}|\nabla \eta|/\epsilon))\overline{u}_s^{m\alpha}. \end{equation}

Moreover, by the $\Delta_2$ condition
\begin{align*}
\Phi(\overline{u}|\nabla \eta|/\epsilon) &\le \max\{(|\nabla \eta|/\epsilon)^\ell,(|\nabla \eta|/\epsilon)^m\}\Phi(\overline{u}) \\ &= g_1(x,\epsilon)\Phi(\overline{u}),
\end{align*}
where $g_1(x,\epsilon)=\max\{(|\nabla \eta|/\epsilon)^\ell,(|\nabla \eta|/\epsilon)^m\}$. Again, by the $\Delta_2$ condition

\begin{align*}\tilde{\Phi}(\epsilon\eta^{m-1}\phi(|\nabla \overline{u}||\nabla \overline{u}|)) &\le \max\{(\epsilon\eta^{m-1} )^{\ell'},(\epsilon\eta^{m-1})^{m'}\}\tilde{\Phi}(|\nabla \overline{u}||\nabla \overline{u}|) \\ &= g_2(x,\epsilon) \tilde\Phi(\phi(|\nabla \overline{u}|)|\nabla \overline{u}|),\end{align*}

where $g_2(x,\epsilon)=\max\{(\epsilon\eta^{m-1} )^{\ell'},(\epsilon\eta^{m-1})^{m'}\}$. We remember that $\tilde{\Phi}(\phi(t)t)\le C\Phi(t)\le C\phi(t)t^2$ for some constant $C$, thus, from (\ref{M3}) and the last inequalities we conclude

 \begin{align}\label{M4}
 m\int \eta^{m-1}(\overline{u}_s^{m\alpha}\overline{u}-k^{m\alpha+1})\phi(|\nabla \overline{u}|)\nabla \overline{u}\nabla \eta \nonumber\le & \int m[ \tilde{\Phi}(\epsilon\eta^{m-1}\phi(|\nabla \overline{u}||\nabla \overline{u}|)) \\ &+\Phi(\overline{u}|\nabla \eta|/\epsilon)]\overline{u}_s^{m\alpha},\nonumber \\   \le & \int  Cm[g_2(x,\epsilon) \phi(|\nabla \overline{u}|)|\nabla \overline{u}|^2\nonumber \\& +g_1(x,\epsilon)\phi(\overline{u})\overline{u}^2]\overline{u}_s^{m\alpha}.\end{align}

 We combine (\ref{M2}) with (\ref{M4}) to obtain that

 \begin{equation}\begin{split}  m\alpha\int \eta^m\overline{u}_s^{m\alpha-1}\overline{u}&\phi(|\nabla \overline{u}_s|)|\nabla \overline{u}_s|^2+\int(\eta^m-Cmg_2(x,\epsilon)) \overline{u}_s^{m\alpha}\phi(|\nabla \overline{u}|)|\nabla \overline{u}|^2\\ & \le \int Cmg_1(x,\epsilon)\phi(\overline{u})\overline{u}^2\overline{u}_s^{m\alpha}+\int |f|\eta^m\overline{u}_s^{m\alpha}\overline{u} . \end{split}\label{M5}\end{equation}

Note that for small $\epsilon$, $g_2(x,\epsilon)= \epsilon^{m'} \eta^{(m-1)m'}=\epsilon^{m'} \eta^{m}$, which implies also that $\eta^m-Cmg_2(x,\epsilon)=\eta^m -Cm\epsilon^{m'} \eta^{m}=\eta^m(1-Cm\epsilon^{m'})$ and thus, we can choose a small $\epsilon$ in such a way that $\eta^m-Cmg_2(x,\epsilon)>0$, hence, by fixing such a $\epsilon$ and combining with the fact that $\overline{u}_s\le \overline{u}$, $|\nabla \overline{u}_s|\le |\nabla \overline{u}|$, we conclude from (\ref{M5}) that

 \begin{equation}\begin{split}  (m\alpha+1-Cm\epsilon^{m'})\int \eta^m\overline{u}_s^{m\alpha}\phi(|\nabla \overline{u}_s|)|\nabla \overline{u}_s|^2 \le& \int Cmg_1(x,\epsilon)\phi(\overline{u})\overline{u}^2\overline{u}_s^{m\alpha} \\ &+\int |f|\eta^m\overline{u}_s^{m\alpha}\overline{u} . \end{split}\label{M6}\end{equation}

 Now, we shall study the term with the $f$ function in (\ref{M6}). We begin by noting that $\overline{u}\ge k$ implies that $\overline{u}^{m-1}\ge k^{m-1}$, whence, $\overline{u}\le \overline{u}^m/k^{m-1}$, therefore, by using the Holder, Young and interpolation inequalities we have that

 \begin{align}
 \int |f|\eta^m\overline{u}_s^{m\alpha}\overline{u} &\le \|f\|_q\|\eta^m\overline{u}_s^{m\alpha}\overline{u}\|_{q'},\nonumber \\
 &\le \frac{\|f\|_q}{k^{m-1}}\|\eta^m\overline{u}_s^{m\alpha}\overline{u}\|_{q'}, \nonumber\\
 &\le \frac{\|f\|_q}{k^{m-1}}\|\eta\overline{u}_s^{\alpha}\overline{u}\|_m^{m(1-N/mq)}\|\eta\overline{u}_s^{\alpha}\overline{u}\|_{m^\star}^{mN/mq}\label{M7} \\
 &\nonumber \le \frac{\|f\|_q}{k^{m-1}\delta^{mq/(mq-N)}}\int \eta^r\overline{u}_s^{m\alpha}\overline{u}^m+\frac{\|f\|_q\delta^{mq/N}}{k^{m-1}}\|\eta\overline{u}_s^{\alpha}\overline{u}\|_{m^\star}^m,
\end{align}
where $\delta>0$ will be chosen later and $m^\star$ denotes the Sobolev critical exponent $mN/(N-m)$. From (\ref{M6}) and (\ref{M7}) we have that

\begin{equation}\begin{split}  (m\alpha+&1-Cm\epsilon^{m'})\int \eta^m\overline{u}_s^{m\alpha}\phi(|\nabla \overline{u}_s|)|\nabla \overline{u}_s|^2 \le \int Cmg_1(x,\epsilon)\phi(\overline{u})\overline{u}^2\overline{u}_s^{m\alpha} \\ &+\frac{\|f\|_q}{k^{m-1}\delta^{mq/(mq-N)}}\int \eta^m\overline{u}_s^{m\alpha}\overline{u}^m+\frac{\|f\|_q\delta^{mq/N}}{k^{m-1}}\|\eta\overline{u}_s^{\alpha}\overline{u}\|_{m^\star}^m . \end{split}\label{M8}\end{equation}

To proceed, notice that $|\nabla(\eta \overline{u}_s^{\alpha+1})|^m\le C(\overline{u}_s^{m(\alpha+1)}|\nabla \eta|^m+\eta^m\overline{u}_s^{m\alpha}|\nabla\overline{u}_s|^m)$ and by using (\ref{equim}) and the $\Delta_2$ condition, we infer that

\begin{equation}
\label{M9}\int |\nabla(\eta \overline{u}_s^{\alpha+1})|^m\le C\int(\overline{u}_s^{m(\alpha+1)}|\nabla \eta|^m+\eta^m\overline{u}_s^{m\alpha}\phi(|\nabla\overline{u}_s|)|\nabla\overline{u}_s|^2).\end{equation}

We use the Sobolev embedding, (\ref{M8}) and (\ref{M9}) to get the following inequality

\begin{equation}\label{M10}\begin{split} &\|\eta\overline{u}_s^{\alpha+1}\|_{m^\star}^m\le  C\int\overline{u}_s^{m(\alpha+1)}|\nabla \eta|^m+\frac{C}{m\alpha+1-Cm\epsilon^{m'}}\int mg_1(x,\epsilon)\phi(\overline{u})\overline{u}^2\overline{u}_s^{m\alpha} \\ & +\frac{C}{m\alpha+1-Cm\epsilon^{m'}}\left(\frac{\|f\|_q}{k^{m-1}\delta^{mq/(mq-N)}}\int \eta^m\overline{u}_s^{m\alpha}\overline{u}^m+\frac{\|f\|_q\delta^{mq/N}}{k^{m-1}}\|\eta\overline{u}_s^{\alpha}\overline{u}\|_{m^\star}^m\right).\end{split}
\end{equation}

We let $s\to \infty$ and use the monotone convergence theorem to conclude from (\ref{M10}) that

\begin{equation}\label{M11}\begin{split} &\|\eta\overline{u}^{\alpha+1}\|_{m^\star}^m-\frac{C}{m\alpha+1-Cm\epsilon^{m'}}\frac{\|f\|_q\delta^{mq/N}}{k^{m-1}}\|\eta\overline{u}^{\alpha+1}\|_{m^\star}^m\le  C\int\overline{u}^{m(\alpha+1)}|\nabla \eta|^m\\ &+\frac{C}{m\alpha+1-Cm\epsilon^{m'}}\left(\int mg_1(x,\epsilon)\overline{u}^{m(\alpha+1)} +\frac{\|f\|_q}{k^{m-1}\delta^{mq/(mq-N)}}\int \eta^m\overline{u}^{m(\alpha+1)}\right).\end{split}
\end{equation}

Choose $$\delta= \left(\frac{(m\alpha+1-Cm\epsilon^{m'})k^{m-1}}{2C\|f\|_q}\right)^{N/mq},$$

\noindent and substitute it on (\ref{M11}) to find that

\begin{equation}\label{M12}\begin{split} &\|\eta\overline{u}^{\alpha+1}\|_{m^\star}^m\le C\int\overline{u}^{m(\alpha+1)}|\nabla \eta|^m+\frac{C}{m\alpha+1-Cm\epsilon^{m'}}\int mg_1(x,\epsilon)\overline{u}^{m(\alpha+1)} \\ & +C\left(\frac{\|f\|_q}{(m\alpha+1-Cm\epsilon^{m'})k^{m-1}}\right)^{N/(mq-N)}\int \eta^m\overline{u}^{m(\alpha+1)}.\end{split}
\end{equation}

If $$G(R_1,R_2)= \left(\frac{1}{R_1-R_2}\right)^\ell+\left(\frac{1}{R_1-R_2}\right)^m,$$

\noindent we conclude from (\ref{M12}) and the definition of $\eta$ and $g_1$ that

\begin{equation}\label{M13}\begin{split} \|\eta\overline{u}^{\alpha+1}&\|_{m^\star}^m\le \frac{C}{(R_1-R_2)^m}\int_{\Omega_{R_2}}\overline{u}^{m(\alpha+1)}+\frac{CG(R_1,R_2)}{m\alpha+1-Cm\epsilon^{m'}}\int_{\Omega_{R_2}} \overline{u}^{m(\alpha+1)} \\ & +C\left(\frac{\|f\|_q}{(m\alpha+1-Cm\epsilon^{m'})k^{m-1}}\right)^{N/(mq-N)}\int_{\Omega_{R_2}}\overline{u}^{m(\alpha+1)}.\end{split}
\end{equation}

Now, we use the fact that $\eta=1$ in $\Omega_{R_1}$ and (\ref{M13}) to infer that

\begin{equation}\label{M14}\begin{split} \|\overline{u}^{\alpha+1}&\|_{L^{m^\star}(\Omega_{R_1})}\le C\left[\frac{1}{R_1-R_2}+\left(\frac{G(R_1,R_2)}{m\alpha+1-Cm\epsilon^{m'}}\right)^{1/m}\right]\|\overline{u}^{\alpha+1}\|_{L^{m}(\Omega_{R_2})}\\ &+C\left(\frac{\|f\|_q}{(m\alpha+1-Cm\epsilon^{m'})k^{m-1}}\right)^{N/m(mq-N)}\|\overline{u}^{\alpha+1}\|_{L^{m}(\Omega_{R_2})}.\end{split}
\end{equation}

Let $\chi=N/(N-m)$. It follow from (\ref{M14}) that

\begin{equation}\label{M15}\begin{split} \|\overline{u}&\|_{L^{m\chi(\alpha+1)}(\Omega_{R_1})}\le C^{1/(\alpha+1)}\Bigg[\frac{1}{(R_1-R_2)}+\left(\frac{G(R_1,R_2)}{m\alpha+1-Cm\epsilon^{m'}}\right)^{1/m}\\ &+\left(\frac{\|f\|_q}{(m\alpha+1-Cm\epsilon^{m'})k^{m-1}}\right)^{N/m(mq-N)}\Bigg]^{1/(\alpha+1)}\|\overline{u}\|_{L^{m(\alpha+1)}(\Omega_{R_2})}.\end{split}
\end{equation}

For $n\in \{0,1,2,\cdots\}$, let $\alpha+1=\chi^n$ and $R_n=R_2+(R_1-R_2)/2^n$. From (\ref{M15}) it follows that

\begin{equation}\label{M16}\begin{split} \|\overline{u}&\|_{L^{m\chi^{n+1}}(\Omega_{R_{n+1}})}\le C^{1/\chi^n}\Bigg[\frac{1}{R_{n}-R_{n+1}}+\left(\frac{G(R_{n},R_{n+1})}{m(\chi^n-1)+1-Cm\epsilon^{m'}}\right)^{1/m}\\ &+\left(\frac{\|f\|_q}{(m(\chi^n-1)+1-Cm\epsilon^{m'})k^{m-1}}\right)^{N/m(mq-N)}\Bigg]^{1/\chi^n}\|\overline{u}\|_{L^{m\chi^n}(\Omega_{R_n})}.\end{split}
\end{equation}

Once $R_n-R_{n+1}=(R_1-R_2)/2^{n+1}$ goes to zero as $n$ goes to infinity, there is a positive constant $C$ (independent of $n$) such that $$G(R_{n},R_{n+1})^{1/m}\le \frac{C}{(R_n-R_{n+1})},$$

\noindent therefore, from (\ref{M16}) and the last inequality, we obtain

\begin{equation*}\label{M17}\begin{split} \|\overline{u}&\|_{L^{m\chi^{n+1}}(\Omega_{R_{n+1}})}\le C^{1/\chi^n}\Bigg[\frac{2^{n+1}}{R_{1}-R_{2}}+\left(\frac{2^{n+1}}{R_{1}-R_{2}}\frac{1}{m(\chi^n-1)+1-Cm\epsilon^{m'}}\right)^{1/m}\\ &+\left(\frac{\|f\|_q}{(m(\chi^n-1)+1-Cm\epsilon^{m'})k^{m-1}}\right)^{N/m(mq-N)}\Bigg]^{1/\chi^n}\|\overline{u}\|_{L^{m\chi^n}(\Omega_{R_n})}.\end{split}
\end{equation*}

By taking $n=0,1,\cdots$, we note that $\|\overline{u}\|_{L^{m\chi^n}(\Omega_{R_n})}$ is finite for every $n$. Moreover, because $\chi>1$, there is a $n_0$, which will depend upon $\chi$, such that the following holds $$\left(\frac{1}{m(\chi^n-1)+1-Cm\epsilon^{m'}}\right)^{1/m\chi^n}\le 1,\ \forall\ n\ge n_0,$$

and $$\left(\frac{1}{m(\chi^n-1)+1-Cm\epsilon^{m'}}\right)^{N/m(mq-N)\chi^n}\le 1,\ \forall\ n\ge n_0,$$

\noindent whence,

\begin{equation}\label{M18}\begin{split} \|\overline{u}\|_{L^{m\chi^{n+1}}(\Omega_{R_{n+1}})}\le C^{1/\chi^n}\left[\frac{2^{n+1}}{R_{1}-R_{2}}+\left(\frac{2^{n+1}}{R_{1}-R_{2}}\right)^{1/m}+\left(\frac{\|f\|_q}{k^{m-1}}\right)^{\beta}
\right]^{1/\chi^n}\|\overline{u}\|_{L^{m\chi^n}(\Omega_{R_n})},\end{split}
\end{equation}
\noindent holds for any $n \geq n_0$ where $\beta=N/m(mq-N)$. We can also assume that for $n\geq n_0$,

$$\left(\frac{2^{n+1}}{R_{1}-R_{2}}\right)^{1/m}\le \frac{2^{n+1}}{R_{1}-R_{2}},$$

\noindent and $$1\le\frac{2^{n+1}}{R_{1}-R_{2}},$$

\noindent therefore, from (\ref{M18}), we conclude that

\begin{equation*}\label{M19}\begin{split} \|\overline{u}\|_{L^{m\chi^{n+1}}(\Omega_{R_{n+1}})}\le \left[\frac{C}{R_1-R_2}\left(1+\frac{\|f\|_q^\beta}{k^{\beta(m-1)}}\right)\right]^{1/\chi^n}2^{(n+1)/\chi^n}\|\overline{u}\|_{L^{m\chi^n}(\Omega_{R_n})},\ \forall\ n\ge n_0,\end{split}
\end{equation*}

\noindent which implies that (after an argument of iteration) for $n\geq n_0$

\begin{equation*}\begin{split} \|\overline{u}\|_{L^{m\chi^{n+1}}(\Omega_{R_{n+1}})}\le \prod_{i=n_0}^n\left[\frac{C}{R_1-R_2}\left(1+\frac{\|f\|_q^\beta}{k^{\beta(m-1)}}\right)\right]^{1/\chi^i}2^{(i+1)/\chi^i}\|\overline{u}\|_{L^{m\chi^{n_0}}(\Omega_{R_{n_0}})}.\end{split}
\end{equation*}

By letting $n\to\infty$ we infer that

\begin{equation}\label{M20}\begin{split} \|\overline{u}\|_{L^\infty(\Omega_{R_1})}\le \left[\frac{C}{R_1-R_2}\left(1+\frac{\|f\|_q^\beta}{k^{\beta(m-1)}}\right)\right]^{1/(1-\chi)}\|\overline{u}\|_{L^{m\chi^{n_0}}(\Omega_{R_{n_0}})}.\end{split}
\end{equation}

We conclude from (\ref{M20}) that $u^+\in L^\infty (\Omega_{R_1})$. By a similar argument, we also have that $u^-\in L^\infty (\Omega_{R_1})$. To extend the result to the boundary, for small $s>0$, let $U_s=\{x\in \mathbb{R}^N\setminus\overline{\Omega}:\ \operatorname{dist}(x,\partial\Omega)<s\}$. Let $\Omega_s=\overline{\Omega}\cup U_s$. Let $\tilde{u},\tilde{f}:\Omega_s\to\mathbb{R}$ be extensions by zero of $u,f$. Note that $\tilde{u}$ is a solution of the problem (\ref{p1}) with $\tilde{f}$ in the place of $f$ and $\Omega_s$ in the place of $\Omega$. Now we apply the same argument as before to conclude that $u\in L^\infty(\Omega)$.
\end{proof}

\section{The proof our main theorems}
In this section we shall give the proof of our main theorems using the Orlicz-Sobolev framework discussed in previous sections.
\subsection{The proof of Theorem \ref{th1}}
First of all, using Proposition \ref{teo_prin} we know that problem \eqref{p1} admits at least one solution in $\w$. According Theorem
\ref{taux} we mention that $u$ is in $L^{\infty} (\Omega)$ whenever $\Phi$ is equivalent to the function  $t \rightarrow |t|^r, \,~r\in (1,\infty))$ and $\ell > 1$. Besides that, using Proposition \ref{unic}, we know that problem \eqref{p1} admits at most one weak solution. This ends the proof.

\subsection{The proof of Theorem \ref{th2}}
Initially, using hypothesis $(\phi_{3})$ and Proposition \ref{l0}, we obtain at least one weak solution $u \in \w$. Furthermore, using $(\phi_{3})^{\prime}$ instead of $(\phi_{3})$, we obtain two weak solutions $u_{1}, u_{2} \in \w$ in such way $J(u_{1}), J(u_{2}) > 0$, see Proposition \ref{l1}. Hence, taking the negative part of $u_{1}$ as testing function, we deduce that $u_{1} \geq 0$. At the same time, using the positive part of $u_{2}$, we observe that $u_{2} \leq 0$ in $\Omega$. Now, using Theorem \ref{taux}, we know also that $u_{1}, u_{2}$ is in $L^{\infty} (\Omega)$ whenever $\ell > 1$ whenever the function $\Phi$ is equivalent to $t \rightarrow |t|^{r}$ for some $r > 1$. So that regularity results on quasilinear elliptic problems imply that $u_{1}, u_{2}$ are in $C^{1}(\overline{\Omega})$, see Lieberman \cite{Lie1,Lie2} whenever $\ell > 1$. This finishes the proof.

\end{document}